\documentclass[12pt]{article}
\usepackage{amssymb}
\usepackage{amsmath,amsthm}
\usepackage{enumerate}
\usepackage{tikz}

 \newtheorem{remark}{Remark}

 \newtheorem{lemma}[remark]{Lemma}
 \newtheorem{theorem}[remark]{Theorem}
 
 \newtheorem{corollary}[remark]{Corollary}

\title{The local metric dimension of strong product graphs}

\author{Gabriel A. Barrag\'{a}n-Ram\'{\i}rez and Juan A. Rodr\'{\i}guez-Vel\'{a}zquez
    \\
{\small Departament d'Enginyeria Inform\`atica i Matem\`atiques,}\\
{\small Universitat Rovira i Virgili,}  {\small Av. Pa\"{\i}sos
Catalans 26, 43007 Tarragona, Spain.} \\{\small
 gbrbcn\@@gmail.com, juanalberto.rodriguez\@@urv.cat}
}
\begin{document}
\maketitle

\begin{abstract}
A vertex $v\in V(G)$ is said to distinguish two vertices $x,y\in V(G)$ of a nontrivial connected graph $G$ if the distance from $v$ to $x$ is different from  the distance from $v$ to $y$.
 A set $S\subset V(G)$ is a \emph{local metric generator} for $G$ if every two adjacent vertices of $G$ are distinguished by some vertex
of $S$. A local metric generator with the minimum cardinality is called a \emph{local metric
basis} for $G$ and its cardinality, the \emph{local metric dimension} of $G$. 
 It is known that the problem of computing the
local metric dimension of a graph is NP-Complete.
In this paper we study the problem of finding exact values or bounds for the local metric dimension of strong product of graphs.
\end{abstract}


\section{Introduction}
A \textit{ metric generator} of a metric space $(X,d)$ is a set $S\subset X$ of points in the space  with the property that every point of $X$  is uniquely determined by the distances from the elements of $S$. 
The \textit{metric dimension}
$\dim(X)$ of $(X, d)$ is the smallest integer $t$ such that there is a metric generator  of cardinality $t$. A metric generator of 
cardinality $\dim(X)$ is called a\textit{ metric basis} of $X$.

The concept of metric dimension of a general metric space first appeared in 1953 in \cite{MR0268781},  but it attracted a little attention, except for the case of graphs. 
Given a simple and connected graph $G=(V,E)$, defined on the vertex set $V$ and the edge set $E$, we consider the function $d_G:V\times V\rightarrow \mathbb{N}\cup \{0\}$, where $d_G(x,y)$ is the length of a shortest path between $u$ and $v$ and $\mathbb{N}$ is the set of positive integers. It is readily seen that $(V,d_G)$ is a metric space. 

The notion of metric dimension of a graph was introduced by Slater in \cite{Slater1975}, where the metric generators were called \emph{locating sets}. Harary and Melter independently introduced the same concept in  \cite{Harary1976}, where metric generators were called \emph{resolving sets}. Applications of this invariant to the navigation of robots in networks are discussed in \cite{Khuller1996} and applications to chemistry in \cite{Johnson1993,Johnson1998}.  This invariant was studied further in a number
of other papers including, for instance \cite{Bailey2011, Caceres2007, Chartrand2000, Feng20121266, Guo2012raey, Haynes2006, Melter1984, Saenpholphat2004, Yero2011}.  Several variations of metric generators including resolving dominating sets \cite{Brigham2003}, independent resolving sets \cite{Chartrand2003}, local metric sets \cite{Okamoto2010}, strong resolving sets \cite{Sebo2004}, $k$-metric generators \cite{Estrada-Moreno2013}, simultaneous metric generators \cite{Ramirez2014}, etc. have since been introduced and studied.

In this article we are interested in the study  of local metric generators, also called local metric sets \cite{Okamoto2010}. A set $S$ of vertices
in a connected graph $G$ is a \emph{local metric generator} for $G$ if every two adjacent vertices of $G$ are distinguished by some vertex
of $S$, \textit{i.e.}, for every $u,v\in V(G)$ there exists $s\in S$ such that $d_G(u,s)\ne d_G(v,s)$. A local metric generator with the minimum cardinality is called a \emph{local metric
basis} for $G$ and its cardinality, the \emph{local metric dimension} of G, is  denoted by $\dim_l(G)$.
  The following main results were obtained in \cite{Okamoto2010}.

  \begin{theorem}{\rm \cite{Okamoto2010}} \label{The1Zhang} Let $G$ be a nontrivial connected graph of order $n$. Then
   $\dim_l(G)=n-1$ if and only if $G$ is  complete, and $\dim_l(G)=1$ if and only if $G$ is bipartite.
\end{theorem}

The clique number $\omega(G)$ of a graph $G$ is the order of a largest complete subgraph in $G$.

  \begin{theorem}{\rm \cite{Okamoto2010}} \label{The1Zhang2} Let $G$ be connected graph of order $n$. Then
   $\dim_l(G)=n-2$ if and only if $\omega(G)=n-1$.
\end{theorem}

The local metric dimension of graphs has been previously studied in 
\cite{BarraganRamirez201427,Rodriguez-Velazquez-Fernau2013,MR3218546,Okamoto2010,Rodriguez-Velazquez2013LDimCorona,Rodriguez-Velazquez-et-al-2014}.
In particular, it was shown in  
\cite{Rodriguez-Velazquez-Fernau2013,MR3218546}
that the problem of computing the local metric  dimension is
NP-Complete.  This suggests
finding the strong metric dimension for special classes of graphs or obtaining good bounds on this
invariant.   In this paper we study the problem of finding exact values or sharp
bounds for the
local metric dimension of strong product graphs.

We begin by giving some basic concepts and notations. For two adjacent vertices $u$ and $v$ of $G=(V,E)$ we use the notation  $u\sim v$ and for two isomorphic graphs $G $ and $G'$ we use $G\cong G'$. For a  vertex
$v$ of $G$, $N_G(v)$ denotes the set of neighbors that $v$ has in $G$, {\it i.e.,}  $N_G(v)=\{u\in V:\; u\sim v\}$. The set $N_G(v)$ is called the \emph{open neighborhood of} $v$ in $G$  and $N_G[v]=N_G(v)\cup \{v\}$ is called the \emph{closed neighborhood of} $v$ in $G$.  

We will use the notation $K_n$, $K_{r,s}$, $C_n$, $N_n$ and $P_n$ for complete graphs,  complete bipartite graphs, cycle graphs, empty graphs and path graphs, respectively.

 The strong product of two graphs $G=(V_1,E_1)$ and $H=(V_2,E_2)$ is the graph $G\boxtimes H=(V,E)$, such that $V=V_1\times V_2$ and two vertices $(a,b),(c,d)\in V$ are adjacent in $G\boxtimes H$ if and only if 
\begin{itemize}
\item[] $a=c$ and $bd\in E_2$, or
\item[] $b=d$ and $ac\in E_1$, or
 \item[] $ac\in E_1$ and $bd\in E_2$.
\end{itemize}
We would point out  that the Cartesian product $G\square  H$ is a subgraph of $G\boxtimes H$ and for complete graphs  $K_r\boxtimes K_s=K_{rs}$.  

One of our tools will be a well-known result, which states the relationship between the vertex distances in $G\boxtimes H$ and the vertex distances in the factor graphs.
\begin{remark}{\rm \cite{Hammack2011}} Let $G$ and $H$ be two connected graphs. Then
$$d_{G\boxtimes H}((a,b),(c,d))=\max \{d_G(a,c), d_H(b,d)\}.$$ 
\end{remark}
For the remainder of the paper, definitions will be introduced whenever a concept is needed.

\section{General Bounds}
\label{Bounds}

We begin by giving   general  bounds for the local metric dimension of strong product graphs.

\begin{theorem}\label{generalUpperBound}
Let $G$ and $H$ be two connected graphs of order $n_1\ge 2$ and $n_2\ge 2$, respectively. Then $$3 \le \dim_l(G\boxtimes H)\le n_1\cdot \dim_l(H) + n_2\cdot \dim_l(G) - \dim_l(G)\cdot \dim_l(H).$$
\end{theorem}

\begin{proof}
Let $V_1$ and $V_2$ be the set of vertices of $G$ and $H$, respectively. 
We claim that  $S=(V_1\times S_2)\cup (S_1\times V_2)$ is a local metric generator for $G\boxtimes H$, where $S_1$ and $S_2$ are local metric basis for $G$ and $H$, respectively.

 Let $(u_i,v_j), (u_k,v_l) \in V_1\times V_2-S$ be two adjacent vertices of $G\boxtimes H$. 
  If $i=k$, then $v_j$ and $v_l$ are adjacent in $H$ and there exists $b\in S_2$ such that 
  $d_{G\boxtimes H}((u_i,b),(u_i,v_j))=d_H(b,v_j)\ne d_H(b,v_l)=d_{G\boxtimes H}((u_i,b),(u_k,v_l))$. So, 
  $(u_i,v_j)$ and $(u_k,v_l)$ are distinguished by  $(u_i,b) \in (V_1\times S_2)\subset S$. Analogously, if $j=l$, then $u_i$ and $u_k$ are adjacent in $G$ and there exists $a\in S_1$ such that 
  $d_G(a,u_i)\ne d_G(a,u_k)$ and, as above,
 $(u_i,v_j)$ and $(u_k,v_l)$ are distinguished by  $(a,v_j)\in (S_1\times V_2)\subset S$. Finally, if 
 $u_iu_k\in E_1$ and $v_jv_l\in E_2$, then for any  
 $a\in S_1$ such that 
  $d_G(a,u_i)\ne d_G(a,u_k)$ we have 
  $$d_{G\boxtimes H}((u_i,v_j),(a,v_j)) = d_{G}(u_i,a) \ne  d_{G}(u_k,a) =\max\{d_{G}(u_k,a),1\} =d_{G\boxtimes H}((a,v_j),(u_k,v_l)).$$
Thus, $(u_i,v_j)$ and $(u_k,v_l)$ are distinguished by  $(a,v_j)\in S_1\times V_2\subset S$. Then we conclude that  $S$ is a local metric generator for $G\boxtimes H$ and, as a consequence, 
$\dim_l(G\boxtimes H)\le |S|= n_1\cdot \dim_l(H) + n_2\cdot \dim_l(G) - \dim_l(G)\cdot \dim_l(H)$.

To prove the lower bound, let $B$ be  a local metric basis 
of $G\boxtimes H$. Given   $(u_1,v_1)\in B$,
chose $u^*\in N_G(u_1)$, $v^*\in N_H(v_1)$ and
define $$W=\{(u^*,v_1), (u_1,v^*), (u^*,v^*)\}.$$
Since $(u_1,v_1)$ is not able to distinguish any pair  of adjacent vertices in $W$, there exists
$(u_2,v_2)\in B - \{(u_1,v_1)\}$.
Let $$q=\min_{(a,b)\in W}\{d_{G\boxtimes H}((u_2,v_2),(a,b))\}.$$
Now, as $d_{G\boxtimes H}((a,b),(u_2,v_2))\in \{q, q+1\}$ for every
$(a,b)\in W$, by  Dirichlet's box principle, there are two vertices
$(x_1,y_1),(x_2,y_2)\in W$ such that $$d_{G\boxtimes H}((u_2,v_2),(x_1,y_1))= d_{G\boxtimes H}((u_2,v_2),(x_2,y_2)).$$ Hence,
$B-\{(u_1,v_1),(u_2,v_2)\}\ne \emptyset$, 
and the result follows.
\end{proof}

Since $K_{n_1}\boxtimes K_{n_2} \cong K_{n_1\cdot n_2}$ and for any complete graph $K_n$, $\dim_l(K_n)=n-1$, we deduce
$$
\dim_l(K_{n_1}\boxtimes K_{n_2})  = n_1\cdot n_2 - 1  = n_1\cdot \dim_l(K_{n_2}) + n_2\cdot \dim_l(K_{n_1}) - \dim_l(K_{n_1})\cdot \dim_l(K_{n_2}).
$$

Therefore,  the upper bound is tight. Examples of non-complete graphs, where the upper bound is attained, can be derived from Theorem \ref{Generalizacompleto-por-G}.

In order to show that the lower bound is tight, consider two  paths $P_t$ and $P_{t'}$,  where $t'\le t\le 2t'-1 $,   $V(P_t)=\{u_1,u_2,\dots,u_t\}$ and  $u_i\sim u_{i+1}$, for every $i\in \{1,\dots , t-1\}$.
 Also, take $v_1,v_{t'}\in V(P_{t'})$ such that $d_{P_{t'}}(v_1,v_{t'})=t'-1$. 
It is not difficult to check that $\{(u_1,v_1),(u_{t'},v_{t'}),(u_t,v_1)\}$ is a local metric generator for $P_t\boxtimes P_{t'}$, so that  Theorem \ref{generalUpperBound} leads to  $\dim_l(P_t\boxtimes P_{t'})=3$.

\section{The Particular Case of Adjacency $k$-Resolved Graphs}
 
Now we will give some results involving the diameter or the radius of $G$. The \textit{eccentricity} $\epsilon(v)$ of a vertex $v$ in a connected graph $G$ is the maximum distance between $v$ and any other vertex $u$ of $H$. So, the \textit{diameter} of $G$ is defined as $$D(G)=\displaystyle\max_{v\in V(G)}\{\epsilon(v)\},$$ while the \textit{radius}  is defined as $$r(G)=\displaystyle\min_{v\in V(G)}\{\epsilon(v)\}.$$ 

Given  two vertices $x$ and $y$ in a connected graph $G=(V,E)$, the interval $I[x, y]$ between $x$ and $y$ is defined as the collection of all vertices which lie on some shortest $x-y$ path. Given a nonnegative integer $k$, we say that \emph{$G$ is adjacency $k$-resolved} if for every two adjacent vertices $x,y\in V$, there exists $w\in V$ such that 

\begin{itemize}
\item[]$d_G(y,w)\ge k$ and $x \in  I[y,w]$, or 

\item[] $d_G(x,w)\ge k$ and  $y \in  I[x,w]$. 
\end{itemize}

For instance, the path and the cycle graphs of order $n$ ($n\ge 2$) are adjacency $\left\lceil \frac{n}{2}\right\rceil$-resolved, the two-dimensional grid graphs $P_r\square P_t$ are adjacency $\left(\lceil\frac{r}{2}\rceil+\lceil\frac{t}{2}\rceil\right)$-resolved, and the hypercube graphs $Q_k$ are adjacency $k$-resolved.

\begin{theorem} \label{resolvedGraphs}
Let $H$ be an adjacency $k$-resolved graph of order $n_2$ and let $G$ be a non-trivial graph of diameter $D(G)<k$. Then
$\dim_l(G\boxtimes H)\le n_2\cdot \dim_l(G)$.
\end{theorem}

\begin{proof}
Let $V_1=\{u_1,u_2,...,u_{n_1}\}$ and $V_2=\{v_1, v_2,...,v_{n_2}\}$ be the set of vertices of $G$ and $H$, respectively. Let $S_1$ be a local metric generator for $G$. We will show that $S=S_1\times V_2$ is a local metric generator for $G\boxtimes H$. Let $(u_i,v_j),(u_r,v_l)$ be two adjacent vertices of $G\boxtimes H$. We differentiate the following two cases.

\noindent{Case 1.} $j=l$. Since  $u_i \sim u_r$ and  $S_1$ is a local metric generator for $G$,  there exists $u\in S_1$ such that $d_G(u_i,u)\ne d_G(u_r,u)$. Hence,
$$d_{G\boxtimes H}((u_i,v_j),(u,v_j))=d_G(u_i,u)\ne d_G(u_r,u)=d_{G\boxtimes H}((u_r,v_j),(u,v_j)).$$

\noindent{Case 2.} $v_j\sim v_l$. Since $H$ is adjacency $k$-resolved, there exists $v\in V_2$ such that $(d_H(v,v_l)\ge k$ and $v_j \in  I[v,v_l])$  or $(d_H(v,v_j)\ge k$ and $v_l\in  I[v,v_j])$. Say $d_H(v,v_l)\ge k$ and $v_j \in  I[v,v_l]$.
In such a case, as $D(G)<k$, for every $u\in S_1$ we have
\begin{align*}
d_{G\boxtimes H}((u_i,v_j),(u,v))&=\max\{d_G(u_i,u), d_H(v_j,v)\}\\
& < d_H(v,v_l) \\
&=\max\{d_G(u,u_r), d_H(v,v_l)\}\\
&=d_{G\boxtimes H}((u_r,v_l),(u,v)).
\end{align*}
Therefore, $S$ is a local metric generator for $G\boxtimes H$.
\end{proof}

\begin{lemma}\label{Lemma-k-resolved-bipartite}
Let $H$ be a  connected bipartite graph  of order greater than or equal to three. Then $H$ is adjacency $k$-resolved for any $k\in \{2,..,r(H)\}$.
\end{lemma}
\begin{proof}
Let   $x,y,w\in V(H)$ such that $x\sim y$ and  $d_H(x,w)=k$, for some $k\in \{2,..,r(H)\}$. Since $H$ does not have cycles of odd length, $d_H(w,y)\ne k$. Thus, either $d_H(w,y)=d_H(w,x)+d_H(x,y)= k+1$ or $d_H(w,x)=d_H(w,y)+d_H(y,x)= k.$ Therefore, the result follows.
\end{proof}  
Now we derive a consequences of combining Theorem \ref{resolvedGraphs} and Lemma \ref{Lemma-k-resolved-bipartite}.

\begin{theorem} \label{UpperBoundsH-bipartite}
Let $G$ and $H$ be two connected non-trivial graphs. If $H$ is bipartite and  $D(G)<r(H)$, then
$\dim_l(G\boxtimes H)\le |V(H)|  \dim_l(G)$.
\end{theorem}

As we will show in Theorem \ref{Th-dim=n(n1-t1)}, the above inequality is tight.

\section{The Role of True Twin Equivalence Classes}
 Two  vertices $u$ and $v$ of a graph $G$ are 
 true twins if $N_G[u]=N_G[v]$.
Note that if two vertices $u$ and $v$ of a graph $G$ are true twins, then  $d_G(x,u)=d_G(x,v)$, for every $x\in V(G)-\{u,v\}$.
We define the \textit{true twin equivalence relation} $\mathcal{R}$ on $V(G)$ as follows:
$$x \mathcal{R} y \longleftrightarrow N_G[x]=N_G[y].$$
If the true twin equivalence classes are $U_1,U_2,...,U_t$, then every local metric generator of $G$ must contain at least $|U_i|-1$ vertices from $U_i$, for each $i\in \{1,...,t\}$. 
 Thus the following result presented in  \cite{Okamoto2010} holds.
 
  \begin{theorem}{\rm \cite{Okamoto2010}} \label{ThPrevioClasses}
  If $G$ is a nontrivial connected graph of order $n$ having $t$ true
twin equivalence classes, then $\dim_l(G) \ge  n - t.$
  \end{theorem}

\begin{figure}[h]
\begin{center}
\begin{tikzpicture}
[inner sep=0.7mm, place/.style={circle,draw=black!30,fill=black!30,thick},xx/.style={circle,draw=black!99,fill=black!99,thick},
transition/.style={rectangle,draw=black!50,fill=black!20,thick},line width=1pt,scale=0.5]

\coordinate (X) at (9,0 ); \coordinate (Y) at (13,0 ); \coordinate (Z) at (17,0 );
\coordinate (X1) at (7.5,2.5 );  \coordinate (X3) at (10.5,2.5 );
\coordinate (Y1) at (11.5,-2.5 );  \coordinate (Y3) at (14.5,-2.5 );
\coordinate (Z1) at (15.5,2.5 );  \coordinate (Z3) at (18.5,2.5 );

\draw[black!40] (X)--(Y)--(Z);
\draw[black!40] (X3)--(X1)--(X)--(X3);

\draw[black!40] (Y1)--(Y)--(Y3);

\draw[black!40] (Z1)--(Z3)--(Z)--(Z1);

\node at (X) [place]  {$3$};\node at (X1) [place]  {$1$};\node at (X3) [place]  {$2$};
\node at (Y) [place]  {$4$};
{};\node at (Y1) [place]  {$5$};\node at (Y3) [place]  {$6$};
\node at (Z) [place]  {$7$};
{};\node at (Z1) [place]  {$8$};\node at (Z3) [place]  {$9$};

\end{tikzpicture}
\end{center}
\caption{This graph has $t=7$ true twin equivalence classes; two of them are $\{1,2\}$ and $\{8,9\}$ and the remain classes are singleton sets. A local metric basis is $\{1,9\}$ while a metric basis is $\{1,5,9\}$. Thus, $\dim_l(G)=n-t=2<3=\dim(G)$.}\label{Fig}
\end{figure}
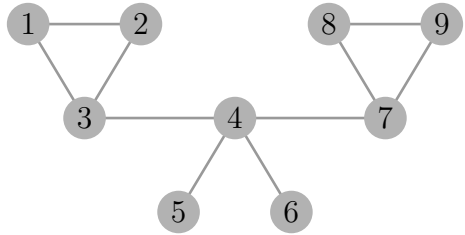


Note that the complete graph has only one true twin equivalence class and in any triangle-free graph all the true twin equivalence classes are singleton. 
As an example of non-complete graph $G$ of order $n$ having $t$ true twin equivalence classes, where $\dim_l(G)=n-t$, we take $G=K_1+\left(\displaystyle\bigcup_{i=1}^lK_{r_i}\right)$, $r_i\ge 2$, $l\ge 2$. In this case $G$ has $t=l+1$ true twin equivalence classes, $n=1+\sum_{i=1}^lr_i$ and $\dim_l(G)=\sum_{i=1}^l(r_i-1)=n-t$. Figure \ref{Fig} shows another example of  graph where the bound given in Theorem \ref{ThPrevioClasses} is reached.

\begin{lemma}\label{lemmaClasses}
Let $G$ and $H$ be two non-trivial connected graphs of order $n_1$ and $n_2$, having $t_1$ and $t_2$ true twin equivalent classes, respectively. Then the vertex set of $G\boxtimes H$ is partitioned into $t_1t_2$ true twin equivalent classes.
\end{lemma}

\begin{proof}
First of all, we would point out that for any $a\in V(G)$ and $b\in V(H)$ it holds
$$N_{G\boxtimes H}[(a,b)]= \{(x,y): \; x\in N_G[a], y\in N_H[b]\}
=N_G[a] \times N_H[b].$$ 

Now, since the result immediately holds for complete graphs, we assume that $G\not\cong K_{n_1}$ or $H\not\cong K_{n_2}$. Let $U_1,U_2,...,U_{t_1}$ and $U'_1,U'_2,...,U'_{t_2}$ be the true twin equivalence classes of $G$ and $H$, respectively.
Since  each  $U_i$ (and $U'_j$) induces  a clique and its vertices have identical closed neighbourhoods, for every $a,c\in U_i$ and $b,d\in U'_j$,
$$N_{G\boxtimes H}[(a,b)] 
=N_G[a] \times N_H[b] 
 =N_G[c]\times N_H[d] 
 =N_{G\boxtimes H}[(c,d)].
$$
Hence, $V(G)\times V(H)$ is partitioned as $V(G)\times V(H)=\bigcup_{j=1}^{t_2}\left(\bigcup_{i=1}^{t_1} U_i\times U'_j \right)$,
where   $U_i\times U'_j$  induces a clique in $G\boxtimes H$ and its vertices have identical closed neighbourhoods.
Moreover, for any $(a,b)\in U_i\times U'_j$ and $(c,d)\in U_k\times U'_l$, where $i\ne k$ or $j\ne l$, we have 
$$N_{G\boxtimes H}[(a,b)] =N_G[a] \times N_H[b]
\ne N_G[c]\times N_H[d]
=N_{G\boxtimes H}[(c,d)].
$$
Therefore, the true twin equivalence classes of $G\boxtimes H$ are of the form $U_i\times U'_j$, where $i\in \{1,..,t_1\}$ and $j\in \{1,..,t_2\}$.
\end{proof}

We would point out that the above result was indirectly obtained in \cite{Rodriguez-Velazquez-et-al2014}, proof of Theorem 2.3.

Theorem \ref{ThPrevioClasses} and Lemma \ref{lemmaClasses} directly lead to the next result.  

\begin{theorem}\label{Generalizacompleto-por-G}
Let $G$ and $H$ be two non-trivial connected graphs of order $n_1$ and $n_2$, having $t_1$ and $t_2$ true twin equivalence classes, respectively. Then $$\dim_l(G\boxtimes H)\ge n_1n_2-t_1t_2.$$
\end{theorem}

By Theorems \ref{The1Zhang}, \ref{generalUpperBound} and   \ref{Generalizacompleto-por-G} we deduce the following result. 
\begin{theorem}\label{F-Consequences}
Let $G$ and $H$ be two non-trivial connected graphs of order $n_1$ and $n_2$, having $t_1$ and $t_2$ true twin equivalence classes, respectively. Then the following assertions hold:
\begin{enumerate}[{\rm (i)}]
\item If $\dim_l(G)=n_1-t_1$ and $\dim_l(H)=n_2-t_2$, then $\dim_l(G\boxtimes H)= n_1n_2-t_1t_2$.
\item If  $\dim_l(G)=n_1-t_1$ and $H$ is bipartite, then $n_2(n_1-t_1)\le \dim_l(G\boxtimes H)\le n_2(n_1-t_1)+t_1$.
\end{enumerate}
\end{theorem}

Since  any complete graph $K_n$ has only one true twin equivalence class, Theorem \ref{F-Consequences} leads to the next result.

\begin{corollary}\label{KnporOtros}
Let $H$ be a  connected graph  of order $n'\ge 2$   having $t$   true twin equivalent classes. Then for any   integer $n\ge 2$,
$$\dim_l(K_n\boxtimes H)=nn'-t.$$
In particular, if $H$ does not have   true twin vertices, then
$$\dim_l(K_n\boxtimes H)=n'(n-1).$$
\end{corollary}

Note that if  $H$ is an adjacency $k$-resolved graph, for $k\ge 2$, then $H$ does not have true twin vertices. Therefore,  
Theorems \ref{Generalizacompleto-por-G} and \ref{resolvedGraphs}  lead to the following result.

\begin{theorem}\label{Th-dim=n(n1-1)}
Let $H$ be an adjacency $k$-resolved graph of order $n_2$ and let $G$ be a non-trivial connected graph of order $n_1$, having $t_1$ true twin equivalence classes and  diameter $D(G)<k$. If
 $\dim_l(G)=n_1-t_1$, then
$\dim_l(G\boxtimes H)= n_2(n_1-t_1)$.
\end{theorem}

Our next result can be deduced from Corollary \ref{Lemma-k-resolved-bipartite} and Theorem \ref{Th-dim=n(n1-1)}  or from Theorems \ref{Generalizacompleto-por-G} and \ref{UpperBoundsH-bipartite}.

\begin{theorem}\label{Th-dim=n(n1-t1)}
Let $H$ be connected bipartite graph of order $n_2$ and let $G$ be a non-trivial connected graph of order $n_1$, having $t_1$ true twin equivalence classes. If
 $\dim_l(G)=n_1-t_1$ and $D(G)<r(H)$, then
$\dim_l(G\boxtimes H)= n_2(n_1-t_1)$.
\end{theorem}

\section{The Particular Case of  $P_t\boxtimes G$}

In this section we assume that  $t$ is an integer greater than or equal to two and $V(P_t)=\{u_1,u_2,\dots,u_t\}$, where $u_i\sim u_{i+1}$, for every $i\in \{1,\dots , t-1\}$. In the proof of the next lemma we will use the notation ${\cal B}_r(x)$ for the closed ball of center $x\in V(G)$ and radius $r\ge 0$, \textit{i.e}.,
$${\cal B}_r(x).=\{y\in V(G):\; d_G(x,y)\le r\}.$$

\begin{lemma}\label{los extremos}
Let $G$ be a connected graph  and let $t\ge 1$ be an integer. Let  $ u_{i_1},   u_{i_2}, \ldots , u_{i_b}$ be the first components of the elements in a local metric basis of $P_t\boxtimes G$,  where $ {i_1} \le   {i_2}\le \cdots \le {i_b}$. Then the following assertions hold.
\begin{enumerate}[{\rm (i)}]
\item $i_2\le D(G)+1$ and $i_{b-1} \ge t- D(G)$.

\item For any $l\in \{1,\dots, b-2\}$, $i_{l+2}\le 2D(G)+i_l$. 

\item $i_3\le 2D(G)+1$. 
\end{enumerate}
 \end{lemma}

\begin{proof}
Let $B$ be a local metric basis of $P_t\boxtimes G$ and let $ u_{i_1},   u_{i_2}, \ldots , u_{i_b}$ be the first components of the elements in  $B$,  where $ {i_1} \le   {i_2}\le \cdots \le {i_b}$. First of all, notice that $|B|=b$ and, by Theorem \ref{generalUpperBound}, $b\ge 3$.  

We first proceed to prove (i). Suppose, for the contrary, that $i_2 > D(G)+1$. Let $y,z\in V(G)$ such that $(u_{i_1},y)\in B$ and $z\in N_G(y)$.
 If $i_1\ne 1$, then  no vertex in $B$ is able to 
distinguish  $(u_1,y)$ and $(u_1,z)$. Now, if $i_1=1$,
then no vertex in $B$ is able to 
distinguish $(u_{2},y)$ and $(u_2,z)$. 
So, in both cases we get a contradiction.
The proof of $i_{b-1}\ge t-D(G)$ is deduced by symmetry. Hence, (i) follows.

To prove (ii) we proceed  by contradiction. Suppose that $i_{l+2}> 2D(G)+i_l$ for some $l\in \{1,\dots, b-2\}$. In such a case we have that  $i_{l+1}> D(G)+i_l$ or  $i_{l+2}> D(G)+i_{l+1}$. We suppose that $i_{l+1}> D(G)+i_l$, being the second case analogous. We now take $y,z\in V(G)$ such that $(u_{i_{l+1}},y)\in B$ and $z\in N_G(y)$. Notice that $(u_{i_l+D(G)},y)$ and $(u_{i_l+D(G)},z)$ are adjacent.We differentiate the following cases
 for   $(u_{i_k},w)\in B$. If $k\le l$, then $i_l+D(G)-i_k\ge D(G)$ and so
$$d_{P_t\boxtimes G}((u_{i_k},w),(u_{i_l+D(G)},y))=i_l+D(G)-i_k=d_{P_t\boxtimes G}((u_{i_k},w),(u_{i_l+D(G)},z)).$$
If $k= l+1$ and $i_{l+1}\ne i_{l+2}$,  then  $w=y$ and since  $i_{l+1}> D(G)+i_l$, we have
$$d_{P_t\boxtimes G}((u_{i_k},w),(u_{i_l+D(G)},y))=i_k-i_l-D(G)=d_{P_t\boxtimes G}((u_{i_k},w),(u_{i_l+D(G)},z)).$$ 
If $k= l+1$ and $i_{l+1}=i_{l+2}$,  then from the assumption   $i_{l+2}> 2D(G)+i_l$  we have that $i_k-i_l-D(G)>D(G)$ and so 
$$d_{P_t\boxtimes G}((u_{i_k},w),(u_{i_l+D(G)},y))=i_k-i_l-D(G)=d_{P_t\boxtimes G}((u_{i_k},w),(u_{i_l+D(G)},z)).$$
If $k\ge   l+2$, then the assumption   $i_{l+2}> 2D(G)+i_l$ leads to $i_k-i_l-D(G)>D(G)$ and so 
$$d_{P_t\boxtimes G}((u_{i_k},w),(u_{i_l+D(G)},y))=i_k-i_l-D(G)=d_{P_t\boxtimes G}((u_{i_k},w),(u_{i_l+D(G)},z)).$$ 
Hence, no vertex in $B$ is able to distinguish  $(u_{i_l+D(G)},y)$ from $(u_{i_l+D(G)},z)$, which is a contradiction. Therefore, the proof of (ii) is complete.

Finally, we proceed to prove (iii). If $i_1=1$, then by (ii) we obtain $i_3\le 2D(G)+1$. Hence, we assume that $i_1>1$. For contradiction purposes,  suppose that $i_3> 2D(G)+1$.
  We differentiate two cases for $(u_{i_1},v_1), (u_{i_2},v_2)\in B$. 

\vspace{0.2cm}
\noindent Case 1: $i_1+i_2-2>d_G(v_1,v_2)$.  In this case 
$|{\cal B}_{ i_1-1}(v_1)\cap {\cal B}_{ i_2-1}(v_2)|\ge 2$ and so we take 
$\alpha, \beta \in {\cal B}_{i_1-1}(v_1)\cap {\cal B}_{ i_2-1}(v_2)$
such that $\alpha \sim \beta$.  For the pair of
adjacent vertices $(u_1,\alpha),(u_1,\beta)$  we have 
$$d_{P_t\boxtimes G}((u_{i_1},v_1),(u_1,\alpha))=i_1-1 =
d_{P_t\boxtimes G}((u_{i_1},v_1),(u_1,\beta))$$
and
$$d_{P_t\boxtimes G}((u_{i_2},v_2),(u_1,\alpha))=i_2-1 =
d_{P_t\boxtimes G}((u_{i_2},v_2),(u_1,\beta)).$$
So, neither $(u_{i_1},v_1)$ nor $(u_{i_2},v_2)$ distinguishes $(u_1,\alpha)$ from $(u_1,\beta)$.
Furthermore, for $i_r\ge i_3> 2D(G)+1$ and $(u_{i_r},v_r)\in B$ we have
$$d_{P_t\boxtimes G}((u_{i_r},v_r),(u_1,\alpha))= i_r-1 =d_{P_t\boxtimes G}((u_{i_r},v_r),(u_1,\beta)).$$ 
Therefore, no vertex  $(u_{i_r},v_r)\in B$ distinguishes
$(u_1,\alpha)$ from $(u_1,\beta)$,  which is a contradiction.

\vspace{0.2cm}
\noindent Case 2: $i_1+i_2 -2 \le d(v_{1},v_{2})$.    In this case we have 
$$(D(G)+2-i_1)+(D(G)+2-i_2) =2D(G)+2-(i_1+i_2-2) \ge
2D(G) +2 -d(v_1,v_2) \ge D(G)+2.$$
Hence, there exist $\alpha, \beta \in {\cal B}_{ D(G)+2 -i_1}(v_{ 1})\cap {\cal B}_{D(G)+2-i_2}(v_{ 2})$  such that $\alpha\sim\beta$.  
For the pair of 
adjacent vertices  $(u_{D(G)+2},\alpha),(u_{D(G)+2},\beta)$ we have 
$$d_{P_t\boxtimes G}((u_{i_1},v_{1}),(u_{D(G)+2},\alpha))=D(G)+2-i_1 =d_{P_t\boxtimes G}((u_{i_1},v_{1}),(u_{D(G)+2},\beta))$$
and
$$d_{P_t\boxtimes G}((u_{i_2},v_2),(u_{D(G)+2},\alpha))=D(G)+2 -i_2 =d_{P_t\boxtimes G}((u_{i_2},v_2),(u_{D(G)+2},\beta))$$
So, neither $(u_{i_1},v_1)$ nor $(u_{i_2},v_2)$ distinguishes $(u_{D(G)+2},\alpha)$ from $(u_{D(G)+2},\beta)$.
For $i_r\ge i_3> 2D(G)+1$ and $(u_{i_r},v_r)\in B$ 
we have
$$d_{P_t\boxtimes G}((u_{i_r},v_r),(u_{D(G)+2},\alpha))= i_r-(D(G)+2)=d_{P_t\boxtimes G}((u_{i_r},v_{r}),(u_{D(G)+2},\beta)).$$
Thus, no vertex $(u_{i_r},v_{r})\in B$ distinguishes
$(u_{D(G)+2},\alpha)$ from $(u_{D(G)+2},\beta)$, which  is a contradiction.
\end{proof}

\begin{theorem}\label{long path vs anyGraph}
For any connected $G$ and any integer  $t \ge 2D(G)+1$,
$$\dim_l(P_t\boxtimes G) \ge  \left\lceil\frac{t-1}{D(G)}\right\rceil+1.$$
 \end{theorem}
 
 \begin{proof} 
Let $B$ be a local metric basis of $P_t\boxtimes G$ and let $ u_{i_1},   u_{i_2}, \ldots , u_{i_b}$ be the first components of the elements in  $B$,  where $ {i_1} \le   {i_2}\le \cdots \le {i_b}$.  We differentiate two cases.

\vspace{0.2cm}
\noindent Case 1. $b$ odd. In this case $b-1$ is even and by Lemma \ref{los extremos} (i) and (ii) we have
$$i_2\le D(G)+1,\;i_4\le 3D(G)+1 ,\;\dots ,\; i_{b-1}\le (b-2)D(G)+1.$$ 

\vspace{0.2cm}
\noindent Case 2. $b$ even.  In this case $b-1$ is odd and by Lemma \ref{los extremos}  (iii) and (ii) we have 
$$i_3\le 2D(G)+1, \; i_5\le 4D(G)+1,\; \dots ,\; i_{b-1}\le (b-2)D(G)+1.$$ 
 According to the two cases above and Lemma \ref{los extremos} (i) we have $$t-D(G)\le i_{b-1}\le (b-2)D(G)+1.$$ 
 Therefore,   $b\ge \frac{t-1}{D(G)}+1$.
 \end{proof}

 From now on we say that a set $W\subset V(G\boxtimes H)$ \emph{resolves}   the set $X\subseteq V(G\boxtimes H)$ if every  pair of adjacent vertices in  $ X$ is distinguished by some element in $W$.

\begin{lemma}\label{resolving triads}
Let $G$ and $H$ be two connected nontrivial graphs such that $H$ is bipartite. Let
$u_1,u_2,u_3 \in V(G)$ and  $v_1,v_2\in V(H)$ such that $u_2\in I_G[u_1,u_3]$,
 $d_G(u_1,u_2)\le d_H(v_1,v_2)= D(H)$ and $d_G(u_2,u_3)\ge D(H)$.
Then, for any shortest path $P$ from $u_1$ to $u_2$, the set $B=\{(u_1,v_1),(u_2,v_2),(u_3,v_1)\}$ resolves $V(P)\times V(H)$.
\end{lemma}

\begin{proof}
Let  $P$ be a shortest path form $u_1$ to $u_2$ and let 
$(u_i,v_j),(u_k,v_l)\in 
V(G\boxtimes H)$ be two adjacent vertices 
such that $u_i,u_k\in V(P)$. Without lost of generality, we assume that $d_G(u_i,u_1)\le d_G(u_k,u_1)$. Notice that from this assumption we have that $d_G(u_i,u_3)\ge d_G(u_k,u_3)$.
We differentiate  the following two cases:

\vspace{0.2cm}
\noindent Case 1:  $u_i\sim u_k$.  As $d_G(u_2,u_3)\ge D(H)$ and $u_i,u_k\in V(P)$, we have $D(H)\le  d_G(u_3,u_k) < d_G(u_3,u_i)$ and so 
  $d_{G\boxtimes H}((u_3,v_1),(u_i,v_j))= d_G(u_3,u_i)> d_G(u_3,u_k)
   = d_{G\boxtimes H}((u_3,v_1),(u_k,v_l)).$

\vspace{0.2cm}
\noindent Case 2:  $i=k$. In this case $v_j\sim v_l$ and, as $H$ is a bipartite graph, 
$d_H(v_1,v_j)\ne d_H(v_1,v_l)$ and $d_H(v_2,v_j)\ne d_H(v_2,v_l)$. We assume,
without lost of generality, that $d_H(v_1,v_j)< d_H(v_1,v_l)$.
Notice that 
$$d_H(v_1,v_j)+d_H(v_j,v_2)\ge d_H(v_1,v_2)=D(H) \ge d_G(u_1,u_2)= d_G(u_1,u_i) + d_G(u_i,u_2).$$
Hence, $d_H(v_1,v_j)\ge d_G(u_1,u_i)$ or $d_H(v_j,v_2) > 
d_G(u_2,u_i)$. 
 If $d_H(v_1,v_j)\ge d_G(u_1,u_i)$, then 
$$d_{G\boxtimes H}((u_1,v_1),(u_i,v_j))= d_H(v_1,v_j)<d_H(v_1,v_l) = d_{G\boxtimes H}((u_1,v_1),(u_k,v_l)).$$
\noindent Now, if $d_H(v_j,v_2) > d_G(u_2,u_i)$, then $d_H(v_l,v_2) \ge  d_G(u_2,u_i)=d_G(u_2,u_k)$ and so
 $$d_{G\boxtimes H}((u_2,v_2),(u_i,v_j))= d_H(v_2,v_j)\ne d_H(v_2,v_l) = d_{G\boxtimes H}((u_2,v_2),(u_k,v_l)).$$
 According to the cases above, the result follows.
\end{proof}

\begin{theorem}\label{long path vs bipartite}
For any connected bipartite graph $G$ and any integer  $t \ge 2D(G)+1$,
$$\dim_l(P_t\boxtimes G) = \left\lceil\frac{t-1}{D(G)}\right\rceil+1.$$
 \end{theorem}
\begin{proof}
Let $G$ and $P_t$ be as in the hypotheses. From $\alpha=\left\lfloor \frac{t-1}{D(G)}\right\rfloor$ and two diametral vertices  $a,b\in V(G)$
we define a set $B_{\alpha}$ as follows.  

If $\alpha=\frac{t-1}{D(G)}$,  then 
$$B_{\alpha}=\{(u_1,a), (u_{D(G)+1},b),(u_{2D(G)+1},a),(u_{3D(G)+1},b),\dots,(u_{\alpha D(G)+1},b)\}$$ 
for  $\alpha$ is odd and
$$B_{\alpha}=\{(u_1,a), (u_{D(G)+1},b),(u_{2D(G)+1},a),(u_{3D(G)+1},b),\dots,(u_{\alpha D(G)+1},a)\}$$ 
for   $\alpha$  even.

If $\alpha<\frac{t-1}{D(G)}$,  then 
$$B_{\alpha}=\{(u_1,a), (u_{D(G)+1},b),(u_{2D(G)+1},a),(u_{3D(G)+1},b),\dots,(u_{\alpha D(G)+1},b),(u_t,a)\}$$ 
for  $\alpha$   odd and
$$B_{\alpha}=\{(u_1,a), (u_{D(G)+1},b),(u_{2D(G)+1},a),(u_{3D(G)+1},b),\dots,(u_{\alpha D(G)+1},a),(u_t,b)\}$$ 
for  $\alpha$   even.
We would point out that, in any case, $|B_{\alpha}|= \left\lceil\frac{t-1}{D(G)}\right\rceil + 1 $.

We will show that   
$B_{\alpha}$ is a local metric generator for $P_t\boxtimes G$.
In order to see that, let $(u_i,v_j)$ and  
$(u_k,v_l)$ be two adjacent   vertices belonging to  $V(P_t\boxtimes G)- B_{\alpha}$. We consider,  without lost of generality, that  $i\le k$ and we differentiate  the following three cases for $k$.
\begin{itemize}
\item $1\le k\le D(G)+1$. Let $T_1=\{u_1,\dots ,u_{D(G)+1}\}\times V(G)$. 
In this case   $(u_i,v_j),(u_k,v_l)\in T_1$
 and, by Lemma \ref{resolving triads} the set $\{(u_1,a),(u_{D(G)+1},b),(u_{2D(G)+1},a)\}\subset B_{\alpha}$
resolves  $T_1$.

\item $pD(G)+2\le k\le (p+1)D(G)+1$, for some integer $p\in \{1,...,\alpha -1\}$.  
Let $T_p=\{u_{pD(G)+1},\dots ,u_{(p+1)D(G)+1}\}\times V(G)$. 
In this case   $(u_i,v_j),(u_k,v_l)\in T_p$
 and we can take $x,y\in \{a,b\}$ so that $X_p=\{(u_{(p-1)D(G)+1 },x),(u_{pD(G)+1},y),(u_{(p+1)D(G)+1},x)\} $ is a subset of $B_{\alpha}$. Thus,  by Lemma \ref{resolving triads} we can conclude that $X_p$
resolves  $T_p$.

\item $\alpha D(G)+2\le k\le t$. 
Let $T_{t}=\{u_{\alpha D(G)+1},\dots ,u_{t}\}\times V(G)$. 
As above,   $(u_i,v_j),(u_k,v_l)\in T_t$
 and we can take $x,y\in \{a,b\}$ so that the set $X_t=\{(u_{(\alpha -1)D(G)+1},x),(u_{\alpha D(G)+1},y),(u_{t},x)\}$  is a subset of $ B_{\alpha}$. Thus,  by Lemma \ref{resolving triads} we can conclude that   $X_t$
resolves  $T_t$.
\end{itemize}
According to the three  cases above we have $\dim_l(P_t\boxtimes G) \le \left\lceil\frac{t-1}{D(G)}\right\rceil+1.$ Therefore, by Theorem \ref{long path vs anyGraph} we conclude the proof.
\end{proof}
 
The authors of \cite{Rodriguez-Velazquez-et-al2014} conjectured that 
for any integers $t$ and $t'$ such that $2\le t'< t$, the metric dimension of $P_{t} \boxtimes P_{t'}$ equals
$\left\lceil\frac{t+t'-2}{t'-1}\right\rceil.$  We are now  able to  prove the conjecture.

\begin{theorem}
For any integers $t$ and $t'$ such that $2\le t'< t$,
$$ \dim(P_{t} \boxtimes P_{t'})= \left\lceil\frac{t+t'-2}{t'-1}\right\rceil.$$
\end{theorem} 

\begin{proof}
As pointed out in Section \ref{Bounds}, for $t'\le t\le 2t'-1 $,  $\dim_l(P_t\boxtimes P_{t'})=3$. Now, 
 since $\dim_l(P_{t} \boxtimes P_{t'})\le \dim(P_{t} \boxtimes P_{t'})$, if $t\ge 2t'-1 $, then by  Theorem \ref{long path vs bipartite} we obtain the lower bound $ \dim(P_{t} \boxtimes P_{t'})\ge  \left\lceil\frac{t+t'-2}{t'-1}\right\rceil.$  The upper bound was obtained in \cite{Rodriguez-Velazquez-et-al2014}. Therefore, the result follows.
\end{proof}

 \section{The Particular Case of  $C_t\boxtimes G$}
 
In this section we assume that  $t$ is an integer greater than or equal to three and $V(C_t)=\{u_1,u_2,\dots,u_t\}$, where $u_1\sim u_t$ and  $u_i\sim u_{i+1}$, for every $i\in \{1,\dots , t-1\}$.

\begin{lemma}\label{los extremosCiclos}
Let $G$ be a connected graph  and let $t\ge 3$ be an integer. Let  $ u_{i_1},   u_{i_2}, \ldots ,  u_{i_b}$ be the first components of the elements in a local metric basis of $C_t\boxtimes G$, where $ {i_1}\le   {i_2}\le \cdots \le {i_b}$.
Then for any $l\in \{1,\dots, b\}$, $d_{C_t}(u_{i_{l+2}},u_{i_l}) \le 2D(G)$, where the subscripts of $i$ are taken modulo $b$. 
\end{lemma}

\begin{proof}
Let  $B$ be a 
local metric basis of $C_t\boxtimes G$ and let $ u_{i_1},   u_{i_2}, \ldots , u_{i_b}$ be the first components of the elements in $B$, where $ {i_1}=1\le   {i_2}\le \cdots \le {i_b}$. First of all, notice that $|B|=b$ and, by Theorem \ref{generalUpperBound}, $b\ge 3$.  

We proceed  by contradiction. Suppose that $d_{C_t}(u_{i_{l+2}},u_{i_l})> 2D(G)$ for some $l\in \{1,\dots, b\}$. In such a case we have that  $d_{C_t}(u_{i_{l+1}},u_{i_l})> D(G)$ or  $d_{C_t}(u_{i_{l+2}},u_{i_{l+1}})> D(G)$. We suppose that $d_{C_t}(u_{i_{l+1}},u_{i_l})> D(G)$, being the second case analogous. We now take $y,z\in V(G)$ such that $(u_{i_{l+1}},y)\in B$ and $z\in N_G(y)$. Notice that $(u_{i_l+D(G)},y)$ and $(u_{i_l+D(G)},z)$ are adjacent. We differentiate the following cases
 for   $(u_{i_k},w)\in B$. If $k\ne l+1$, then $d_{C_t}(u_{i_l+D(G)},u_{i_k})\ge D(G)$ and so
$$d_{C_t\boxtimes G}((u_{i_k},w),(u_{i_l+D(G)},y))=d_{C_t}(u_{i_l+D(G)},u_{i_k})=d_{C_t\boxtimes G}((u_{i_k},w),(u_{i_l+D(G)},z)).$$
If $k= l+1$ and $i_{l+1}\ne i_{l+2}$  then  $w=y$ and since  $d_{C_t}(u_{i_{l+1}},u_{i_l})> D(G)$, we have
$$d_{C_t\boxtimes G}((u_{i_k},w),(u_{i_l+D(G)},y))=d_{C_t}(u_{i_k},u_{i_l+D(G)})=d_{C_t\boxtimes G}((u_{i_k},w),(u_{i_l+D(G)},z)).$$ 
If $k= l+1$ and $i_{l+1}=i_{l+2}$  then from the assumption  $d_{C_t}(u_{i_{l+2}},u_{i_l})> 2D(G)$  we have that $d_{C_t}(u_{i_k},u_{i_l+D(G)})>D(G)$ and so 
$$d_{C_t\boxtimes G}((u_{i_k},w),(u_{i_l+D(G)},y))=d_{C_t}(u_{i_k},u_{i_l+D(G)})=d_{C_t\boxtimes G}((u_{i_k},w),(u_{i_l+D(G)},z)).$$
Hence, no vertex in $B$ is able to distinguish the  adjacent vertices $(u_{i_l+D(G)},y)$ and $(u_{i_l+D(G)},z)$, which is a contradiction. Therefore, the proof is complete.
\end{proof}

\begin{theorem}\label{long Cycle vs anyGraph}
For any connected graph $G$ and any integer  $t\ge 1$,
$$\dim_l(C_t\boxtimes G) \ge  \left\lceil\frac{t}{D(G)}\right\rceil.$$
 \end{theorem}
 
 \begin{proof} 
 If  $3D(G) \ge t \ge 1$, then $ \left\lceil\frac{t}{D(G)}\right\rceil\le 3$ and, by Theorem \ref{generalUpperBound}, the result follows. From now on we take $t> 3D(G)$. 
 Let   $ u_{i_1},   u_{i_2}, \ldots ,  u_{i_b}$ be the first components of the elements in a local metric basis $B$ of $C_t\boxtimes G$, where $ {i_1}=1\le   {i_2}\le \cdots \le {i_b}$.  First of all, notice that  $t+1-i_{b-1}=d_{C_t}(u_{i_1},u_{i_{b-1}})$ and so Lemma \ref{los extremosCiclos} leads to  $i_{b-1}\ge t+1- 2D(G)$.  We now differentiate two cases.

\vspace{0.2cm}
\noindent Case 1. $b$ even.  In this case $b-1$ is odd and by Lemma \ref{los extremosCiclos}  we have 
$$i_3\le 2D(G)+1, \; i_5\le 4D(G)+1,\; \dots ,\; i_{b-1}\le (b-2)D(G)+1.$$ 
Hence, $ t+1- 2D(G) \le  i_{b-1}\le (b-2)D(G)+1$, so that  $b\ge \frac{t}{D(G)}$.

 \vspace{0.2cm}
\noindent Case 2. $b$ odd.  By Lemma \ref{los extremosCiclos}    we have
$$i_3\le D(G)+1,\;i_4\le 3D(G)+1 ,\;\dots ,\; i_{b}\le (b-1)D(G)+1.$$ 
Now, since $t+i_2-i_b=d_{C_t}(u_{i_2},u_{b})\le 2D(G)$, we have $$i_2\le 2D(G)-t+i_b\le (b+1)D(G)-t+1.$$ Hence, 
$$i_2\le (b+1)D(G)-t+1, \; i_4\le (b+3)D(G)-t+1,\; \dots, \; i_{b-1}\le (2b-2)D(G)-t+1.$$
Thus, $t+1-2D(G)\le i_{b-1}\le (2b-2)D(G)-t+1$, so that  $b\ge \frac{t}{D(G)}$.
 \end{proof}
 
 \begin{theorem}\label{long Cycle vs BipartiteGraphs}
For any connected bipartite graph  $G$ and any integer  $t \ge 4D(G)$,
$$\dim_l(C_t\boxtimes G) \le  \left\lceil\frac{t}{D(G)}\right\rceil+1.$$
Furthermore, if $\left\lceil\frac{t}{D(G)}\right\rceil$ is even,
then 
$$\dim_l(C_t\boxtimes G) = \left\lceil\frac{t}{D(G)}\right\rceil.$$
 \end{theorem}
 
 \begin{proof}
Let $G$ and $C_t$ be as in the hypotheses.  From $\alpha=\left\lceil \frac{t}{D(G)}\right\rceil$ and two diametral vertices  $a,b\in V(G)$
we define a set $B_{\alpha}$ as follows.
If $\alpha $ is even, then
$$B_{\alpha}=\{(u_1,a), (u_{D(G)+1},b),(u_{2D(G)+1},a),(u_{3D(G)+1},b),\dots,(u_{(\alpha-1) D(G)+1},b)\}$$ 
and, if $\alpha$  is odd, then 
$$B_{\alpha}=\{(u_1,a), (u_{D(G)+1},b),(u_{2D(G)+1},a),(u_{3D(G)+1},b),\dots,(u_{(\alpha-1) D(G)+1},a),(u_{(\alpha-1) D(G)+1},b)\}.$$

Notice that  $|B_{\alpha}|=\alpha$, for $\alpha $  even, and $|B_{\alpha}|=\alpha+1$, for $\alpha $ odd.  We will show that 
$B_{\alpha}$ is a local metric generator for $C_t\boxtimes G$.
In order to see that, let $(u_i,v_j)$, 
$(u_k,v_l)$ be a pair of adjacent   vertices belonging to  $V(C_t\boxtimes G)- B_{\alpha}$. We consider,  without lost of generality, that  $i\le k$ and we differentiate  the following three cases for $k$.

\begin{itemize}
\item $2\le k\le D(G)+1$. Let $T_1=\{u_1,\dots ,u_{D(G)+1}\}\times V(G)$. 
In this case   $(u_i,v_j),(u_k,v_l)\in T_1$
 and, by Lemma \ref{resolving triads} the set $\{(u_1,a),(u_{D(G)+1},b),(u_{2D(G)+1},a)\}\subset B_{\alpha}$
resolves  $T_1$.

\item $pD(G)+2\le k\le (p+1)D(G)+1$, for some integer $p\in \{1,...,\alpha -2\}$.  
Let $T_p=\{u_{pD(G)+1},\dots ,u_{(p+1)D(G)+1}\}\times V(G)$. 
In this case   $(u_i,v_j),(u_k,v_l)\in T_p$
 and we can take $x,y\in \{a,b\}$ such that $X_p=\{(u_{(p-1)D(G)+1 },x),(u_{pD(G)+1},y),(u_{(p+1)D(G)+1},x)\} $ is a subset of $B_{\alpha}$. Thus,  by Lemma \ref{resolving triads} we can conclude that $X_p$
resolves  $T_p$.

\item $(\alpha-1) D(G)+2\le k\le t+1$. 
Let $T_{t}=\{u_{(\alpha-1) D(G)+1},\dots ,u_{t+1}\}\times V(G)$. 
In this case,   $(u_i,v_j),(u_k,v_l)\in T_t$
 and we take  the set $X_{t}=\{(u_{(\alpha -1)D(G)+1},b),(u_{1},a), (u_{D(G)+1},b)\}\subset B_{\alpha}$. By Lemma \ref{resolving triads} we can conclude that   $X_t$
resolves  $T_t$.
\end{itemize}

According to the three  cases above $B_{\alpha}$ is a local metric generator for $C_t\boxtimes G$ and so $\dim_l(C_t\boxtimes G) \le |B_{\alpha}|$.
 Therefore, by Theorem \ref{long Cycle vs anyGraph} we conclude the proof.
\end{proof}


\begin{thebibliography}{99}
\expandafter\ifx\csname url\endcsname\relax
  \def\url#1{\texttt{#1}}\fi
\expandafter\ifx\csname urlprefix\endcsname\relax\def\urlprefix{URL }\fi

\bibitem{Bailey2011}
R.~F. Bailey, K.~Meagher, On the metric dimension of grassmann graphs, Discrete
  Mathematics \& Theoretical Computer Science 13~(4) (2011) 97--104.
\newline\urlprefix\url{http://www.dmtcs.org/dmtcs-ojs/index.php/dmtcs/article/view/2049}

\bibitem{BarraganRamirez201427}
G.~Barrag\'{a}n-Ram\'{i}rez, C.~G. G\'{o}mez, J.~A.
  Rodr\'{i}guez-Vel\'{a}zquez, Closed formulae for the local metric dimension
  of corona product graphs, Electronic Notes in Discrete Mathematics 46~(0)
  (2014) 27--34.
\newline\urlprefix\url{http://www.sciencedirect.com/science/article/pii/S1571065314000067}

\bibitem{MR0268781}
L.~M. Blumenthal, Theory and applications of distance geometry, Second edition,
  Chelsea Publishing Co., New York, 1970.

\bibitem{Brigham2003}
R.~C. Brigham, G.~Chartrand, R.~D. Dutton, P.~Zhang, Resolving domination in
  graphs, Mathematica Bohemica 128~(1) (2003) 25--36.
\newline\urlprefix\url{http://mb.math.cas.cz/mb128-1/3.html}

\bibitem{Caceres2007}
J.~C\'{a}ceres, C.~Hernando, M.~Mora, I.~M. Pelayo, M.~L. Puertas, C.~Seara,
  D.~R. Wood, On the metric dimension of cartesian product of graphs, SIAM
  Journal on Discrete Mathematics 21~(2) (2007) 423--441.
\newline\urlprefix\url{http://epubs.siam.org/doi/abs/10.1137/050641867}

\bibitem{Chartrand2000}
G.~Chartrand, L.~Eroh, M.~A. Johnson, O.~R. Oellermann, Resolvability in graphs
  and the metric dimension of a graph, Discrete Applied Mathematics 105~(1-3)
  (2000) 99--113.
\newline\urlprefix\url{http://dx.doi.org/10.1016/S0166-218X(00)00198-0}

\bibitem{Chartrand2003}
G.~Chartrand, V.~Saenpholphat, P.~Zhang, The independent resolving number of a
  graph, Mathematica Bohemica 128~(4) (2003) 379--393.
\newline\urlprefix\url{http://mb.math.cas.cz/mb128-4/4.html}

\bibitem{Estrada-Moreno2013}
A.~Estrada-Moreno, J.~A. Rodr\'{\i}guez-Vel\'{a}zquez, I.~G. Yero, The
  $k$-metric dimension of a graph, Applied Mathematics \& Information Sciences.
  To appear.
\newline\urlprefix\url{http://arxiv.org/abs/1312.6840}

\bibitem{Feng20121266}
M.~Feng, K.~Wang, On the metric dimension of bilinear forms graphs, Discrete
  Mathematics 312~(6) (2012) 1266 -- 1268.
\newline\urlprefix\url{http://www.sciencedirect.com/science/article/pii/S0012365X11005279}

\bibitem{Rodriguez-Velazquez-Fernau2013}
H.~{Fernau}, J.~A. {Rodr{\'{i}}guez-Vel{\'a}zquez}, On the (adjacency) metric
  dimension of corona and strong product graphs and their local variants:
  combinatorial and computational results, arXiv:1309.2275 [math.CO].
\newline\urlprefix\url{http://arxiv-web3.library.cornell.edu/abs/1309.2275}

\bibitem{MR3218546}
H.~Fernau, J.~A. Rodr{\'{\i}}guez-Vel{\'a}zquez, Notions of metric dimension of
  corona products: combinatorial and computational results, in: Computer
  science---theory and applications, vol. 8476 of Lecture Notes in Comput.
  Sci., Springer, Cham, 2014, pp. 153--166.

\bibitem{Guo2012raey}
J.~Guo, K.~Wang, F.~Li, Metric dimension of some distance-regular graphs,
  Journal of Combinatorial Optimization 26 (2013) 190--197.
\newline\urlprefix\url{http://dx.doi.org/10.1007/s10878-012-9459-x}

\bibitem{Hammack2011}
R.~Hammack, W.~Imrich, S.~Klav{\v{z}}ar, Handbook of product graphs, Discrete
  Mathematics and its Applications, 2nd ed., CRC Press, 2011.
\newline\urlprefix\url{http://www.crcpress.com/product/isbn/9781439813041}

\bibitem{Harary1976}
F.~Harary, R.~A. Melter, On the metric dimension of a graph, Ars Combinatoria 2
  (1976) 191--195.
\newline\urlprefix\url{http://www.ams.org/mathscinet-getitem?mr=0457289}

\bibitem{Haynes2006}
T.~W. Haynes, M.~A. Henning, J.~Howard, Locating and total dominating sets in
  trees, Discrete Applied Mathematics 154~(8) (2006) 1293--1300.
\newline\urlprefix\url{http://www.sciencedirect.com/science/article/pii/S0166218X06000035}

\bibitem{Johnson1993}
M.~Johnson, Structure-activity maps for visualizing the graph variables arising
  in drug design, Journal of Biopharmaceutical Statistics 3~(2) (1993)
  203--236, pMID: 8220404.
\newline\urlprefix\url{http://www.tandfonline.com/doi/abs/10.1080/10543409308835060}

\bibitem{Johnson1998}
M.~Johnson, Browsable structure-activity datasets, in: R.~Carb\'{o}-Dorca,
  P.~Mezey (eds.), Advances in Molecular Similarity, chap.~8, JAI Press Inc,
  Stamford, Connecticut, 1998, pp. 153--170.
\newline\urlprefix\url{http://books.google.es/books?id=1vvMsHXd2AsC}

\bibitem{Khuller1996}
S.~Khuller, B.~Raghavachari, A.~Rosenfeld, Landmarks in graphs, Discrete
  Applied Mathematics 70~(3) (1996) 217--229.
\newline\urlprefix\url{http://www.sciencedirect.com/science/article/pii/0166218X95001062}

\bibitem{Melter1984}
R.~A. Melter, I.~Tomescu, Metric bases in digital geometry, Computer Vision,
  Graphics, and Image Processing 25~(1) (1984) 113--121.
\newline\urlprefix\url{http://www.sciencedirect.com/science/article/pii/0734189X84900513}

\bibitem{Okamoto2010}
F.~Okamoto, B.~Phinezy, P.~Zhang, The local metric dimension of a graph,
  Mathematica Bohemica 135~(3) (2010) 239--255.
\newline\urlprefix\url{http://dml.cz/dmlcz/140702}

\bibitem{Ramirez2014}
Y.~Ram\'irez-Cruz, O.~R. Oellermann, J.~A. Rodr\'iguez-Vel\'azquez,
  Simultaneous resolvability in graph families, Electronic Notes in Discrete
  Mathematics 46~(0) (2014) 241 -- 248.
\newline\urlprefix\url{http://www.sciencedirect.com/science/article/pii/S157106531400033X}

\bibitem{Rodriguez-Velazquez2013LDimCorona}
J.~A. {Rodr\'{i}guez-Vel\'{a}zquez}, G.~A. {Barrag\'{a}n-Ram\'{i}rez},
  C.~{Garc\'{i}a G\'{o}mez}, On the local metric dimension of corona product
  graphs, Bulletin of the Malaysian Mathematical Sciences Society. (2015) To
  appear.
\newline\urlprefix\url{http://arxiv-web3.library.cornell.edu/abs/1308.6689}

\bibitem{Rodriguez-Velazquez-et-al-2014}
J.~A. {Rodr\'{i}guez-Vel{\'a}zquez}, C.~{Garc{\'{i}}a G{\'o}mez}, G.~A.
  {Barrag{\'a}n-Ram{\'{i}}rez}, Computing the local metric dimension of a graph
  from the local metric dimension of primary subgraphs, Int. J. Comput. Math.
  92~(4) (2015) 686--693.
\newline\urlprefix\url{http://arxiv.org/abs/1402.0177}

\bibitem{Rodriguez-Velazquez-et-al2014}
J.~A. Rodr\'{i}guez-Vel\'{a}zquez, D.~Kuziak, I.~G. Yero, J.~M. Sigarreta, The
  metric dimension of strong product graphs., Carpathian Journal of Mathematics
  31~(2) (2015) 261--268.

\bibitem{Saenpholphat2004}
V.~Saenpholphat, P.~Zhang, Conditional resolvability in graphs: a survey,
  International Journal of Mathematics and Mathematical Sciences 2004~(38)
  (2004) 1997--2017.
\newline\urlprefix\url{http://www.hindawi.com/journals/ijmms/2004/247096/abs/}

\bibitem{Sebo2004}
A.~Seb\"{o}, E.~Tannier, On metric generators of graphs, Mathematics of
  Operations Research 29~(2) (2004) 383--393.
\newline\urlprefix\url{http://dx.doi.org/10.1287/moor.1030.0070}

\bibitem{Slater1975}
P.~J. Slater, Leaves of trees, Congressus Numerantium 14 (1975) 549--559.

\bibitem{Yero2011}
I.~G. Yero, D.~Kuziak, J.~A. Rodr\'{\i}quez-Vel\'{a}zquez, On the metric
  dimension of corona product graphs, Computers \& Mathematics with
  Applications 61~(9) (2011) 2793--2798.
\newline\urlprefix\url{http://www.sciencedirect.com/science/article/pii/S0898122111002094}

\end{thebibliography}
\end{document}